\documentclass[11pt]{amsart}
\usepackage[margin=1in]{geometry}
\usepackage{amsmath}
\usepackage{amsfonts,amsmath}
\usepackage{paralist}
\usepackage[colorlinks=true]{hyperref}
\hypersetup{urlcolor=blue, citecolor=red}
 \numberwithin{equation}{section}

\newcommand{\dut}{\mbox{$\left(\Delta u\right)^{-3} $}}

\newcommand{\R}{\mathbb{R}}
\newcommand{\RD}{\mathbb{R}^N}

\newcommand{\ets}{(\psi^++\varepsilon)}

\newcommand{\omt}{\Omega_\infty}
\newcommand{\io}{\int_\Omega}

\newcommand{\po}{\partial\Omega }

\newcommand*{\avint}{\mathop{\ooalign{$\int$\cr$-$}}}

\newcommand{\dt}{\Omega}

\newcommand{\idt}{\int_{\dt}}

\newcommand{\ut}{\mbox{$\left(\Delta u_0(x)\right)^{-3} $}}

\newtheorem{theorem}{Theorem}[section]

\newtheorem{clm}{Claim}[section]
\newtheorem{proposition}{Proposition}[section]

\theoremstyle{definition}
\newtheorem{definition}[theorem]{Definition}
\newtheorem{rem}{Remark}

\title[A fourth-order nonlinear parabolic equation in non-divergence form
] 
      {Special solutions to a fourth-order nonlinear parabolic equation in non-divergence form }

\author[Xiangsheng Xu]{}

\subjclass{35D30, 35J66, 35J40, 35K65, 35K41.}
 \keywords{Existence, nonlinear fourth order elliptic equations, degeneracy, 
 	crystal surface	models.\\ {\it Commun. Math. Sci.}, to appear.}

 \email{xxu@math.msstate.edu}

\begin{document}
\maketitle

\centerline{\scshape Xiangsheng Xu}
\medskip
{\footnotesize
 \centerline{Department of Mathematics \& Statistics}
   \centerline{Mississippi State University}
   \centerline{ Mississippi State, MS 39762, USA}
} 

\bigskip


\begin{abstract}
In this paper we study a crystal surface model first proposed by H.~Al Hajj Shehadeh, R.V.~Kohn, and J.~Weare (2011 Physica D, 240,1771-1784). By seeking a solution of a particular function form, we are led to a boundary value problem for a fourth-order nonlinear elliptic equation. The mathematical challenge of the problem is due to the fact that the degeneracy in the equation is directly imposed by one of the two boundary conditions. An existence theorem is established in which a meaningful mathematical interpretation of one of the boundary conditions remains open. Our proof
seems to suggest that this is unavoidable. We also obtain self-similar solutions to the crystal surface model which are positive and unbounded. This is in sharp contrast with the linear biharmonic heat equation.

\end{abstract}


\section{Introduction}\label{sec1}
Let $\Omega$ be a bounded domain in $\RD$ with boundary $\po$.
Consider the initial boundary value problem
\begin{eqnarray}
\partial_t \rho+\rho^2\Delta^2\rho^3 &=& 0\ \  \mbox{in $\Omega_\infty$, } 
\label{cr1}\\
\rho &=&0\ \ \ \mbox{ on $\Sigma_\infty$},\label{cr2}\\
\Delta \rho^3 &=& 0\ \ \ \mbox{ on $\Sigma_\infty$},\label{cr3}\\
\rho|_{t=0}&=&\rho_0  \ \ \ \mbox{on $\Omega$,}\label{cr4}
\end{eqnarray}
where
$\omt= \Omega\times(0,\infty),\ \Sigma_\infty= \partial\Omega\times(0,\infty)$.
If $N=1$, the equation in \eqref{cr1} was proposed by H.~Al Hajj Shehadeh, R.V.~Kohn, and J.~Weare \cite{AKW} as a continuum model for the evolution of a one-dimensional monotone step train separating two facets of a crystal surface in the attachment-detachment-limited regime. In this case, the space variable $x$
is the surface height and $\rho$ the surface slope. Since the surface height is increasing, we expect that
\begin{equation}
\rho\geq 0.
\end{equation} 
The existence of a solution to \eqref{cr1}-\eqref{cr4} was left open in \cite{AKW}.
The mathematical difficulty is due to the boundary condition \eqref{cr2}, which forces the equation  in \eqref{cr1} to be degenerate. As a result, a priori estimates are difficult to obtain. In \cite{GLL2}, 
an existence assertion was established for \eqref{cr1}-\eqref{cr4} with boundary conditions \eqref{cr2} and \eqref{cr3} being replaced by
periodic boundary conditions. 
In \cite{LX2}, the authors reformulated \eqref{cr1} by setting 
\begin{equation}\label{ru1}
\Delta u=\frac{1}{\rho}.
\end{equation}
At least, one can formally show  that $u$ satisfies
\begin{eqnarray}
\partial _tu&=& \Delta\dut  \ \ \ \mbox{in $\Omega_\infty$.}\label{p1}
\end{eqnarray}
This equation was then coupled with the initial boundary conditions 
\begin{eqnarray}
u &=& b_0 (x)\ \ \ \mbox{ on $\Sigma_\infty$},\label{p2}\\
\Delta u &=& b_1 (x)\ \ \ \mbox{ on $\Sigma_\infty$},\label{p3}\\
u(x,0)&=& u_0(x) \ \ \ \mbox{on $\Omega$}\label{p4}
\end{eqnarray}
for given data $b_0(x), \,b_1$, and $u_0(x)$ with properties:
\begin{enumerate}
	\item[(H1)] $b_0(x)\in W^{1,2}(\Omega)\cap  L^\infty(\Omega) $;
	\item[(H2)] $b_1(x)\in W^{2,2}(\Omega)\cap  L^\infty(\Omega)$ and
	$b_1(x)\geq c_0$ a.e. in $\Omega$ for some $c_0>0$ ; 
	\item[(H3)] $u_0(x)\in W^{2,2}(\Omega)$, $\Delta u_0(x)\geq c_1>0$ a.e. in $\Omega$, and $\ut\in W^{2,2}(\Omega)$.
\end{enumerate}
Under these conditions, the existence of a suitably-defined weak solution to \eqref{p1}- \eqref{p4} was obtained in \cite{LX2} for any space dimensions, where it also revealed that there was a singular part in $\Delta u$. That is, one has
\begin{equation}
\Delta u=\frac{1}{\rho}+\nu_s,
\end{equation}
where $\nu_s$ is a non-negative, finite Radon measure.  The function $\rho$ in \eqref{ru1} is also a solution to \eqref{cr1} in a suitable weak sense only if one of the following conditions is met:
\begin{enumerate}
	\item $\rho$ is continuous on $\omt$;
	\item $\nu_s=0$; or
	\item $\rho$  satisfies the additional integrability conditions
	\begin{equation}
	\rho\partial_t\rho\in L^2(0,T;W_0^{1,2}(\Omega)), \ \ \rho^2\in L^2(0,T;W^{2,2}(\Omega))\ \ \mbox{for each $T>0$}.\label{ssd1}
	\end{equation} 
\end{enumerate}
Unfortunately, in multiple space dimensions, none of the above conditions can really be expected. More recently, the authors
in \cite{GM} introduced the change of variable
\begin{equation}
\frac{1}{\rho}=1+v
\end{equation}
and transformed \eqref{cr1} into
\begin{equation}
\partial_tv=\Delta^2\frac{1}{(1+v)^3}
\end{equation}
The equation was then coupled with the initial and periodic boundary conditions. The existence of a ``much stronger'' weak solution than the one in \cite{LX2} was obtained, provided that  the initial data was suitably small,. In particular, the weak solution was shown to decay to $0$ exponentially. 

Thus to the best of our knowledge, no existing work has directly dealt with the boundary condition \eqref{cr2}. In this paper, we shall consider an elliptic version of the problem. Indeed, by seeking a solution of \eqref{cr1}-\eqref{cr3} of the function form
\begin{equation}\label{svs}
\rho(x,t)=A(t) \psi(x),
\end{equation}
we arrive at the following boundary value problem for $\psi$
\begin{eqnarray}
\psi\Delta^2  \psi^3&=& \lambda\ \ \mbox{in $\Omega$,}\label{spr1}\\
\psi&=& 0\ \ \mbox{on $\po$,}\label{spr2}\\
\Delta  \psi^3&=& 0\ \ \mbox{on $\po$,}\label{spr3}\\
\end{eqnarray}
where $\lambda$ is a positive number. (See Section \ref{sec2} for details.) Evidently, the forced degeneracy by the boundary condition \eqref{spr2} is still present in the equation \eqref{spr1}.  For this problem, we have the following
\begin{theorem}\label{exis1}
	Assume that $\Omega$ is a bounded domain in $\RD$ with $C^{2,\alpha}$ boundary $\po$ for some $\alpha>0$. For each $\lambda>0$ there is a function $\psi$ such that
	\begin{enumerate}
		\item[\textup{(C1)}]$\psi\in C^{\infty}_{\textup{loc}}(\Omega)$, $\psi^3\in W^{2,2}(\Omega)$, $\psi(x)>0$ for $x\in \Omega$;
		\item[\textup{(C2)}]$\psi=0$ on $\po$;
		\item[\textup{(C3)}]$\psi(x)\Delta^2\psi^3(x)=\lambda$ for each $x$ in $\Omega$.
	\end{enumerate}
\end{theorem}
The proof of this theorem will be presented in Section \ref{sec2}. Our investigations reveal that it does not seem to be possible to obtain any estimates for $\nabla\Delta\psi^3$. Thus the sense in which the boundary condition \eqref{spr3} is satisfied is an open issue. 
Physically, 
the surface of a crystal below the roughening temperature consists of steps and terraces, 
%
and the ODE describing the evolution of the discrete steps is exactly the finite-difference analogue of problem \eqref{cr1}-\eqref{cr4} \cite{AKW}. Thus the boundary conditions \eqref{cr2} and \eqref{cr3} arise naturally. Obviously, \eqref{spr3} is from \eqref{cr3}. How to bridge the gap here is an interesting open question.

Observe that the function $ \psi$ only needs to satisfy the equation
\begin{equation}\label{8fm1}
\psi^2\Delta^2 \psi^3=\lambda  \psi\ \ \mbox{in $\dt$}
\end{equation} for $A \psi$ to be a solution of \eqref{cr1}. To find a solution to this equation, it seems to be natural to consider the functional
\begin{equation}
H(v)=\frac{1}{6}\idt \left(\Delta v^3\right)^2dx-\frac{\lambda}{2}\idt v^{2}dx \ \ \mbox{on $ W\equiv\{v: v^3\in W^{2,2}(\dt), v^3|_{\po}=0\}$}.
\end{equation}
By the calculations in \eqref{am2} below, we see that the functional is coercive on $W$ for each $\lambda>0$, and hence it has a  minimizer. Unfortunately, $W$ does not seem to be a linear space. As a result, we cannot compute the G\^{a}teaux derivative of this functional. The connection of this minimizer to \eqref{8fm1} is not clear, nor can we ascertain its non-negativity. 

Our solution in \eqref{svs} satisfies the decay condition
\begin{equation}
\|\rho(x,t)\|_{W^{2,2}(\Omega)}\leq \frac{c_1}{(c_2+4\lambda t)^{\frac{1}{4}}},\ \  
\end{equation}
where $c_1, c_2>0$ and $\lambda$ is given as in Theorem \ref{exis1}. We conjecture that this should be true for any solution of problem \eqref{cr1}-\eqref{cr4}.

It is also interesting to seek a self-similar solution of the equation $\partial_t\rho+\rho^2\Delta\rho^3=0$ in $\RD\times (0,\infty)$ of the form
\begin{equation}
\rho(x,t)=t^\alpha f(y), \ \ y=\frac{x}{t^\beta}.
\end{equation}
By the calculations in Section \ref{sec3}, we see that $\alpha=\frac{4\beta-1}{4} $ and $f$ satisfies the equation
\begin{equation}\label{jt1}
f^2(y)\Delta^2f^3(y)-\beta y\cdot\nabla f(y)+\frac{4\beta-1}{4} f(y)=0\ \ \mbox{on $\RD$.}
\end{equation}
If $\beta=0$, then we roughly recover the equation in \eqref{8fm1} in $\RD$.


\begin{definition}
	We say that a function $f$ is a weak solution of \eqref{jt1} if $f^3\in W^{2,2}_{\textup{loc}}(\RD)$ and the equation
	\begin{equation}\label{8t1}
	\int_{ \RD}	\Delta f^3\Delta\left(f^3\xi\right)dy+\frac{\beta}{2}\int_{ \RD}f^2y\cdot\nabla\xi dy+\frac{(4+2N)\beta-1}{4}\int_{ \RD}f^2\xi dy=0
	\end{equation}
	holds for each $\xi\in C_0^\infty(\RD)$.
\end{definition}

To gain some insights into equation \eqref{jt1},  we seek a solution of \eqref{jt1} in the function form
\begin{equation}\label{ss4}
f(y)=cr^s,
\end{equation}
where $c$ is a constant and $ r=|y|$.
A simple calculation shows
\begin{equation}
\nabla r^s=sr^{s-2}y,\ \ \ \Delta r^s=s(s+N-2)r^{s-2}.
\end{equation}
With the aid of this, we plug $f$ in \eqref{ss4} into \eqref{jt1} to derive
\begin{equation}
3s(3s-2)(3s+N-2)(3s+N-4)c^5r^{5s-4}-\beta csr^s+ c\frac{4\beta-1}{4}r^s=0.
\end{equation}
For this to be an identity, we must take
\begin{eqnarray}
s=1,\ \ c^4=\frac{1}{12(N-1)(N+1)}.
\end{eqnarray}
Subsequently, we obtain a non-trivial solution
\begin{equation}\label{exp}
f(y)=\frac{1}{\left(12(N-1)(N+1)\right)^{\frac{1}{4}}}\sqrt{y_1^2+\cdots+y_N^2}.
\end{equation}
That is, no matter what value $\beta$ is, we alway have a positive, unbounded solution to \eqref{jt1}  
in $\RD$. Obviously, nonlinearities in our equation have played a key role. As we recall, the function $f(y)$ in self-similar solutions to the biharmonic heat equation $\partial_tu+\Delta^2u=0$ changes signs infinitely many times and decays to $0$ exponentially as $|y| \rightarrow\infty$ \cite{FGG,GP}. 

If $\beta\geq \frac{1}{4+2N}$ and a weak solution $f$ has the property
\begin{equation}
f^3\in W^{2,2}(\RD),\ \ \ f\in L^2(\RD),
\end{equation}
then $f=0$. This is due to the fact that we can construct a sequence of test functions $\xi_k$ in $C_0^\infty(\RD)$ with the properties
\begin{eqnarray}
\xi_k(y)&=&1\ \ \mbox{on $B_k(0)$},\\
\xi_k(y)&=&0\ \ \mbox{outside $B_{2k}(0)$},\\
|\nabla\xi_k(y)|\leq \frac{c}{k}, &&|\Delta\xi_k(y)|\leq \frac{c}{k^2}\ \ \mbox{on $\RD$.}
\end{eqnarray}
Here and in what follows $B_s(z)$ denotes the ball centered at $z$ with radius $s$ for $z\in \RD$ and $s>0$ and $c$ a positive number. Then we have
\begin{eqnarray}
\Delta(f^3\xi_k)=\xi_k\Delta f^3+2\nabla\xi_k\nabla f^3+f^3\Delta\xi_k&\rightarrow&\Delta f^3\ \ \mbox{strongly in $L^2(\RD)$},\\
\left|\int_{ \RD}f^2y\cdot\nabla\xi_k dy\right|&\leq&c\int_{B_{2k}(0)\setminus B_{k}(0)}f^2dy\rightarrow 0\ \ \mbox{as $k\rightarrow\infty$.}
\end{eqnarray}
Thus let $\xi=\xi_k$ in \eqref{8t1} and take $k\rightarrow\infty$ in the resulting equation to derive the desired result.
\begin{theorem}\label{exis2}Assume 
	\begin{equation}\label{asmp}
	-\frac{1}{4(N-1)}\leq \beta\leq 0.
	\end{equation}
	Then for each pair of positive numbers $c_2, c_4$ there exists a radially symmetric solution $f=f(|y|)=f(r)$ to \eqref{jt1} with the property
	\begin{equation}\label{lb1}
	c_4+c_2r^2\leq f^3(r)\leq c_4+c_2r^2+cr^4\ \ \ \mbox{for some positive number $c=c(N, \beta, c_2)$}.
	\end{equation}	
\end{theorem}

The proof of this theorem will be given in Section \ref{sec3}. Since \eqref{lb1} holds, degeneracy does not occur and  solutions in Theorem \ref{exis2} are very smooth. In addition, they seem to lie in a ``small'' neighborhood of the solution in \eqref{exp}. The existence of any sign-changing weak solutions to \eqref{jt1} remains a open question. 

Self-similar solutions were also studied in \cite{AKW, MN}. They focused on the case  where $\Omega=(0,1)$. Their methods and similarity variables were both different from ours.

Finally, we remark that continuum models for the evolution of a crystal surface have received considerable attention recently. See, for example, \cite{BCF,KDM, MW, X, XX} and the references therein. Mathematical analysis of these models have revealed some very interesting properties of solutions. To mention a few, we refer the reader to \cite{GLLX, LX, LX2} for solutions that contain  measures. The study of exponential decay of solutions can be found in \cite{GM, LS}. Development of singularity and finite extinction of solutions were considered in \cite{GK}. Also see \cite{A} for the existence of analytic solutions.  
\section{Solution by separation of variables}\label{sec2}

We seek a non-trivial solution of \eqref{cr1} of the function form
\begin{equation}
\rho(x,t)=A(t) \psi(x)
\end{equation}
coupled with the boundary conditions
\begin{equation}\label{bc}
\psi=\Delta  \psi^3=0\ \ \po.
\end{equation}
Substitute this into \eqref{cr1} to obtain
\begin{equation}
A^\prime(t) \psi(x)+A^5(t) \psi^2(x)\Delta^2 \psi^3(x)=0.
\end{equation}
If both $A(t)\ne 0$ a.e and $ \psi(x)\ne 0$ a.e., then
\begin{equation}
\frac{A^\prime(t)}{A^5(t)}=- \psi(x)\Delta^2 \psi^3(x).
\end{equation}
This is true if and only if both sides of the equation are a  constant. Denote this constant by $-\lambda$. We obtain
\begin{eqnarray}
A^\prime(t)&=&-\lambda A^5(t),\ \ \ t>0,\label{8ft1}\\
\psi(x)\Delta^2 \psi^3(x)&=&\lambda,\ \ \ x\in\Omega.\label{8ft2}
\end{eqnarray}
Multiplying through \eqref{8ft2} by $ \psi^2$ and integrating over $\Omega$, we  derive, with the aid of \eqref{bc},
\begin{equation}
\lambda\io  \psi^2dx=\io \left(\Delta  \psi^3\right)^2dx.
\end{equation}
Here and in what follows whenever there is no confusion we suppress the dependence of a function on its dependent variables.
Consequently,
\begin{equation}
\lambda \geq 0.
\end{equation}
If $\lambda=0$, then $A(t)=A(0)$ and $\psi$ can be any non-zero constant. The resulting solution is a constant solution of \eqref{cr1}. From here on, we assume
\begin{equation}
\lambda>0.
\end{equation} We
solve \eqref{8ft1} to obtain
\begin{equation}
A(t)=\frac{1}{\left(A^{-4}(0)+4\lambda t\right)^{\frac{1}{4}}}.
\end{equation}
Set
\begin{equation}
v=\Delta  \psi^3.
\end{equation}
This leads to the consideration of the system
\begin{equation}
\left\{\begin{array}{ll}
\Delta v&=\frac{\lambda }{ \psi},\\
\Delta  \psi^3&=v.
\end{array}\right.
\end{equation}
We first consider an approximation to the above system.
\begin{proposition}\label{pbe}
	Let $\Omega$ be a bounded domain in $\RD$ with Lipschitz boundary $\po$ and $\lambda$ a positive number. For $\varepsilon>0$ there exists a pair of functions $(\psi,v)$ such that
	\begin{enumerate}
		\item[\textup{(R1)}]$\psi,v\in W^{1,2}(\Omega)\cap C^{0,\alpha}(\overline{\Omega})$ for some $\alpha\in (0,1)$;
		\item[\textup{(R2)}] $\psi(x)\geq0, v(x)\leq 0$ for each $x\in \Omega$;
		\item[\textup{(R3)}] They satisfy
		the
		boundary value problem
		\begin{eqnarray}
		-\textup{div}\left(3(\psi+\varepsilon)^2\nabla \psi\right)&=&-v \ \ \textup{in $\Omega$,}\\
		-\Delta v&=&-\frac{\lambda}{\psi+\varepsilon} \ \textup{in $\Omega$,}\\
		\psi&=&0\ \ \textup{on $\po$,}\\
		v&=&0\ \ \textup{on $\po$}
		\end{eqnarray}
		in the weak sense.
	\end{enumerate}
\end{proposition}
Later we shall see that we actually have that the strict inequality in (R2) holds.
\begin{proof} We define an operator $T$ from $L^\infty(\Omega)$ into $L^\infty(\Omega)$ as follows: We say $T(g)=\psi$ if $\psi$ is the unique solution of the problem
	\begin{eqnarray}
	-\mbox{div}\left(3(g^++\varepsilon)^2\nabla \psi\right)&=&-v \ \ \textup{in $\Omega$,}\label{ls3}\\
	\psi&=&0\ \ \textup{on $\po$,}\label{ls4}
	\end{eqnarray}
	where $v$ solves
	\begin{eqnarray}
	-\Delta v&=&-\frac{\lambda}{g^++\varepsilon} \ \ \textup{in $\Omega$,}\label{ls1}\\
	v&=&0\ \ \textup{on $\po$.}\label{ls2}
	\end{eqnarray}
	Obviously, $\frac{\lambda}{g^++\varepsilon} \in L^\infty(\Omega)$ and the two equations in \eqref{ls3} and\eqref{ls1} are  both linear and uniformly elliptic. Classical theory \cite{GT} for this type of equations asserts that there is a unique weak solution $v$ to \eqref{ls1}-\eqref{ls2} in the space $W^{1,2}(\Omega)\cap C^{0,\alpha}(\overline{\Omega})$ for some $\alpha\in (0,1)$. This, in turn, implies that problem \eqref{ls3}-\eqref{ls4} has a unique weak solution $\psi$ in the same type of function spaces. That is,  $T$ is well-defined.
	We can further conclude from these relevant a priori estimates that $T$ is also continuous and precomact. 
	To apply the Leray-Schauder fixed point theorem (\cite{GT}, p. 280), we still need to establish that for each $\sigma\in(0,1]$ and each $\psi\in L^\infty(\Omega)$ such that
	\begin{equation}\label{jt2}
	\psi=\sigma T(\psi),
	\end{equation}
	we have 
	\begin{equation}\label{jf3}
	\|\psi\|_{\infty,\Omega}\leq c.
	\end{equation}
	Here and in what follows $\|\cdot\|_{p,\Omega}$ denotes the norm in $L^p(\Omega)$.
	To see this, we observe that \eqref{jt2} is equivalent to  the following equations
	\begin{eqnarray}
	-\mbox{div}\left(3\ets^2\nabla \psi\right)&=&-\sigma v \ \ \textup{in $\Omega$,}\label{jf2}\\
	-\Delta v&=&-\frac{\lambda}{\psi^++\varepsilon} \ \ \textup{in $\Omega$,}\label{jf1}\\
	\psi&=&0\ \ \textup{on $\po$},\\
	v&=&0\ \ \textup{on $\po$}.
	\end{eqnarray}
	Note that the term on the right-hand side of \eqref{jf1} is non-positive. Thus by the maximum principle, we have
	\begin{equation}
	v\leq 0 \ \ \mbox{a.e. in $\Omega$.}
	\end{equation}
	With this in mind, we can
	apply the maximum principle to \eqref{jf2} to obtain
	\begin{equation}
	\psi\geq 0\ \ \mbox{a.e. in $\Omega$.}
	\end{equation}
	Consequently, $\psi^+=\psi$ and we can write \eqref{jf2} as
	\begin{equation}
	\Delta(\psi+\varepsilon)^3=\sigma v\ \ \mbox{a.e. in $\Omega$.}
	\end{equation}
	By the classical uniform estimate for linear elliptic equations, we deduce that for each $p>\frac{N}{2}$ there is a positive number $c=c(N, \Omega)$ such that
	\begin{eqnarray}
	\max_{\Omega}((\psi+\varepsilon)^3-\varepsilon^3)&\leq &c\|v\|_{p, \Omega},\\
	\max_{\Omega}(-v)&\leq &c\left\|\frac{\lambda}{\psi+\varepsilon}\right\|_{p, \Omega}\leq c\frac{\lambda}{\varepsilon}.
	\end{eqnarray}
	Combing the preceding two estimates yields \eqref{jf3}. This completes the proof.
\end{proof}
\begin{proof}[Proof of Theorem \ref{exis1}]
	For each $k\in \{1,2,\cdots\}$
	let $\{\psi_{k}, v_k\}$ be  a solution of the problem
	\begin{eqnarray}
	\Delta \varphi_k&=&v_k\ \ \mbox{in $\Omega$},\label{js1}\\
	\Delta v_{k}&=&\frac{\lambda}{\psi_{k}+\frac{1}{k}} \ \ \mbox{in $\Omega$},\label{js2}\\
	\psi_{k}&=&0\ \mbox{on $\partial \Omega$,}\\
	v_k&=&0\ \ \mbox{on $\partial \Omega$}
	\end{eqnarray}
	in the sense of Proposition \ref{pbe}, where
	\begin{equation}
	\varphi_k= \left(\psi_{k}+\frac{1}{k}\right)^3-\frac{1}{k^3}.
	\end{equation}
	Thus
	we have 
	\begin{equation}
	v_k\leq 0,\ \ \psi_{k}\geq 0\ \ \ \mbox{in $\Omega$.}\label{js4}
	\end{equation}
	We add the term $-v_k$ to both sides of \eqref{js1} and square the resulting equation to derive
	\begin{eqnarray}
	\io\left(\Delta\varphi_k\right)^2dx+\io v_k^2dx&=&2\io \Delta \varphi_kv_k dx
	=-2\io\nabla \varphi_k\nabla v_k dx.\label{js3}
	\end{eqnarray}
	Note  that 
	\begin{equation}
	\varphi_k=0\ \ \mbox{on $\partial \Omega$.}
	\end{equation}Multiply through \eqref{js2} by the term and integrate the resulting equation over $\Omega$ to obtain
	\begin{equation}
	-\io\nabla \varphi_k\nabla v_k dx=\lambda\io\left(\left(\psi_{k}+\frac{1}{k}\right)^2-\frac{1}{k^3\left(\psi_{k}+\frac{1}{k}\right)}\right)dx.
	\end{equation}
	Substitute this into \eqref{js3} to derive
	\begin{equation}\label{am1}
	\io\left(\Delta \varphi_k\right)^2dx+\io v_k^2dx+ 2\lambda\io \frac{1}{k^3\left(\psi_{k}+\frac{1}{k}\right)}dx= 2\lambda\io \left(\psi_{k}+\frac{1}{k}\right)^2dx.
	\end{equation}
	We deduce from Poincar\'{e}'s inequality that 
	\begin{eqnarray}
	\io\varphi_k^2dx&\leq &c\io|\nabla\varphi_k|^2dx\nonumber\\
	&=&c\io\left(\mbox{div}(\varphi_k\nabla\varphi_k)-\varphi_k\Delta\varphi_k\right)dx\nonumber\\
	&=&-c\io\varphi_k\Delta\varphi_kdx\nonumber\\
	&\leq&\frac{1}{2}\io\varphi_k^2dx+c\io|\Delta\varphi_k|^2dx,\label{am2}
	\end{eqnarray}
	from whence follows
	\begin{equation}\label{sobolev}
	\io\varphi_k^2dx\leq c\io|\Delta\varphi_k|^2dx.
	\end{equation}
	With this in mind, we are ready to estimate
	\begin{eqnarray}
	\io \left(\psi_{k}+\frac{1}{k}\right)^2dx&\leq&c\left(\io \left(\psi_{k}+\frac{1}{k}\right)^3dx\right)^{\frac{2}{3}}\nonumber\\
	&=&c\left(\io \varphi_kdx+\frac{|\Omega|}{k^3}\right)^{\frac{2}{3}}\nonumber\\
	&\leq&c\left(\io \varphi_k^2dx\right)^{\frac{1}{3}}+\frac{c}{k^2}\nonumber\\
	&\leq&c\left(\io|\Delta\varphi_k|^2dx\right)^{\frac{1}{3}}+\frac{c}{k^2}.
	\end{eqnarray}
	Use this in \eqref{am1} to obtain
	\begin{equation}
	\io\left(\Delta \varphi_k\right)^2dx+\io v_k^2dx+ 2\lambda\io \frac{1}{k^3\left(\psi_{k}+\frac{1}{k}\right)}dx\leq c.
	\end{equation}
	Since we have assumed that $\po$ is $C^{2, \alpha}$ for some $\alpha>0$, the classical Calder\'{o}n-Zygmund estimate 
	implies that $\{\left(\psi_{k}+\frac{1}{k}\right)^3\}=\{\varphi_k+\frac{1}{k^3}\}$ is bounded
	in $W^{2,2}(\Omega)$. Thus we extract a subsequence of $\{\psi_{k}+\frac{1}{k}\}$, still denoted by $\{\psi_{k}+\frac{1}{k}\}$, such that
	\begin{eqnarray}
	\psi_{k}+\frac{1}{k}&\rightarrow& \psi\ \ \mbox{strongly in $L^2(\Omega)$ and a.e. in $\Omega$},\label{at1}\\
	\left(\psi_{k}+\frac{1}{k}\right)^3&\rightarrow& \psi^3\ \ \mbox{weakly in $W^{2.2}(\Omega)$ and strongly in $W^{1,2}(\Omega)$. }\label{at2}
	\end{eqnarray}
	Similarly, we may assume that
	\begin{equation}\label{am3}
	v_k\rightharpoonup v\ \ \mbox{weakly in $L^2(\Omega)$.}
	\end{equation}
	Now we can take the limit in \eqref{js1} to obtain
	\begin{equation}\label{pse}
	\Delta\psi^3=v\ \ \mbox{in $\Omega$.}
	\end{equation}
	\begin{proposition}\label{pcm}
		The sequence $\{v_k\}$ is bounded in $W^{1.2}_{\textup{loc}}(\Omega)$.
	\end{proposition}
	\begin{proof}
		Let $r>0, z\in \Omega$ be such that 
		\begin{equation}
		B_{r}(z)\subset \Omega.
		\end{equation}
		Choose a cut-off function $\zeta\in C^\infty(\RD)$ with the properties
		\begin{eqnarray}
		\zeta(x)&=&\left\{\begin{array}{ll}
		1&\mbox{if $x\in B_{\frac{r}{2}}(z)$,}\\
		0&\mbox{if $x\in \RD\setminus B_{r}(z)$,}
		\end{array}\right.\\
		0&\leq&\zeta\leq 1,\\
		|\nabla\zeta|&\leq&\frac{c}{r}.
		\end{eqnarray}
		We easily see from \eqref{js2} that
		\begin{equation}\label{am4}
		\Delta(-v_k)\leq 0\ \ \mbox{in $\Omega$}.
		\end{equation}
		That is, $-v_k$ is a non-negative superharmonic function
		in $\Omega$. Since  $v_k$ cannot be identically $0$, the strong maximum principle asserts that
		\begin{equation}
		-v_k(x)>0\ \ \mbox{in $\Omega$.}
		\end{equation}
		Furthermore, we can conclude from  Theorem 8.18 in (\cite{GT}, p.194) 
		that 
		\begin{equation}\label{jw11}
		\inf_{ B_{\frac{r}{2}}(z)}(-v_k(x))\geq c\avint_{ B_{r}(z)}(-v_k(x))dx.
		\end{equation}
		We claim that
		\begin{equation}\label{at3}
		\int_{ B_{r}(z)}(-v(x))dx>0 .
		\end{equation}
		Were this not true, we would have
		\begin{equation}
		v=0\ \ \mbox{a.e. on $B_{r}(z)$.}
		\end{equation}
		We calculate from Fatou's lemma, \eqref{at1},  \eqref{js2}, and \eqref{am3} that
		\begin{eqnarray}
		\int_{ B_{\frac{r}{2}}(z)}\frac{1}{\psi}dx&\leq &\int_{ B_{r}(z)}\frac{\zeta}{\psi}dx\nonumber\\
		&\leq&\liminf_{k\rightarrow\infty}\int_{ B_{r}(z)}\frac{\zeta}{\psi_{k}+\frac{1}{k}}dx\nonumber\\
		&=&\frac{1}{\lambda}\liminf_{k\rightarrow\infty}\int_{ B_{r}(z)}\Delta v_k\zeta dx\nonumber\\
		&=&\frac{1}{\lambda}\liminf_{k\rightarrow\infty}\int_{ B_{r}(z)} v_k\Delta\zeta dx=0.
		\end{eqnarray}
		That is, $\psi=\infty$ on $B_{\frac{r}{2}}(z)$. This contradicts \eqref{at1}. The claim \eqref{at3} follows.
		
		Use $\frac{1}{-v_k}\zeta^2$ as a test function in \eqref{am4} to obtain
		\begin{equation}
		\io\frac{1}{v_k^2}|\nabla v_k|^2\zeta^2dx\leq\io\frac{1}{-v_k}\nabla v_k2\zeta\nabla\zeta dx,
		\end{equation}
		from whence follows
		\begin{equation}\label{jw12}
		\int_{ B_{\frac{r}{2}}(z)}\frac{1}{v_k^2}|\nabla v_k|^2dx\leq cr^{N-2}.
		\end{equation}
		We can easily deduce from \eqref{jw11} and \eqref{at3} that
		there is a positive number $c$ such that 
		\begin{equation}
		\inf_{ B_{\frac{r}{2}}(z)} (-v_k)\geq c\ \ \mbox{for $k$ sufficiently large.}
		\end{equation}
		This together with \eqref{jw12} implies
		that
		\begin{equation}
		\int_{ B_{\frac{r}{2}}(z)}|\nabla v_k|^2dx\leq cr^{N-2}\ \ \mbox{for $k$ sufficiently large.}.
		\end{equation}
		Since this is true for each $r>0$ and each $z\in \Omega$ with $B_r(z)\subset\Omega$,  the proposition follows.\end{proof}

	To continue the proof of Theorem \ref{exis1}, we see from
	the proposition that $v\in W^{1,2}_{\mbox{loc}}(\Omega)$. This along with the fact that $-v$ is superharmonic in $\Omega$ asserts that 
	\begin{equation}\label{at7}
	\inf_{ B_{\frac{r}{2}}(z)}(-v)\geq c\int_{ B_{r}(z)}(-v)dx>0.
	\end{equation}
	We see from \eqref{pse} that $\psi^3$ is also superharmonic in $\Omega$. Thus there holds
	\begin{equation}\label{at11}
	\inf_{ B_{\frac{r}{2}}(z)}\psi^3\geq c\int_{ B_{r}(z)}\psi^3dx\ \ \mbox{for some $c>0$}.
	\end{equation}
	We can claim that
	\begin{equation}\label{at12}
	\int_{ B_{r}(z)}\psi^3dx>0\ \ \mbox{for each $r>0$ and each $z\in\Omega$ with $B_r(z) \subset\Omega$}
	\end{equation}
	Were this not true, we would have
	\begin{equation}
	\psi=0 \ \ \ \mbox{in $B_r(z)$ for some $r>0, z\in \Omega$ with $B_r(z) \subset\Omega$.}
	\end{equation}
	By \eqref{pse}, we also have that $v=0$ on the same ball. This contradicts \eqref{at7}.
	Obviously, if we replace $\psi$ by $\psi_{k}$ in\eqref{at11}, the resulting inequality still holds. This combined with  \eqref{at12} implies that 
	\begin{equation}\label{8as2}
	\psi_{k}\geq c\ \ \mbox{on $B_{\frac{r}{2}}(z)$ for some $c>0$.}
	\end{equation}
	Hence we can pass to the limit in \eqref{js2} to get
	\begin{eqnarray}\label{8as3}
	\Delta v=\frac{\lambda}{\psi}\ \ \ \mbox{in $\Omega$.}
	\end{eqnarray}
	This, along with  \eqref{at11},  implies that $v$ is locally bounded. With this in mind, we can use  \eqref{pse}  again to conclude that $\psi$ is also locally bounded.
	We  have actually established that for each $r>0, z\in \Omega$ with $B_r(z) \subset\Omega$ there is a positive number $c$ with
	\begin{equation}
	c\leq\psi(x)\leq \frac{1}{c}\ \ \mbox{for each $x\in B_{\frac{r}{2}}(z)$.}
	\end{equation}
	We can conclude  (C1) from a boot strap argument. Take the Laplacian of both sides of \eqref{pse} and substitute \eqref{8as3} into the resulting equation to yields (C3). The proof of Theorem \ref{exis1} is complete. \end{proof}

We would like to point out the negative impact of the boundary condition \eqref{spr2} on a priori estimates. Observe from \eqref{js1} that
\begin{eqnarray*}
	\io\Delta\left(\psi_{k}+\frac{1}{k}\right)^2dx&=&\int_{ \po}2\left(\psi_{k}+\frac{1}{k}\right)\nabla\psi_{k}\cdot\nu d\mathcal{H}^{N-1}\nonumber\\
	&=&\frac{2k}{3}\int_{ \po}3\left(\psi_{k}+\frac{1}{k}\right)^2\nabla\psi_{k}\cdot\nu d\mathcal{H}^{N-1}\nonumber\\
	&=&\frac{2k}{3}\io\Delta\left(\psi_{k}+\frac{1}{k}\right)^3dx=\frac{2k}{3}\io v_k dx\rightarrow-\infty\ \ \mbox{as $k\rightarrow\infty$.}
\end{eqnarray*}
We infer from \eqref{js1} and \eqref{js2} that
\begin{equation}\label{js5}
\io\Delta \left(\psi_{k}+\frac{1}{k}\right)^3\frac{\lambda}{\psi_{k}+\frac{1}{k}}dx=\io v_k\Delta v_k dx=-\io|\nabla v_k|^2dx.
\end{equation}
The left-hand side of the above equation can be calculated as follows:
\begin{eqnarray}
\io\Delta \left(\psi_{k}+\frac{1}{k}\right)^3\frac{\lambda}{\psi_{k}+\frac{1}{k}}dx&=&\io\Delta \left(\psi_{k}+\frac{1}{k}\right)^3\lambda \left(\frac{1}{\psi_{k}+\frac{1}{k}}-k\right)dx\nonumber\\
&&+\lambda k\io\Delta \left(\psi_{k}+\frac{1}{k}\right)^3dx\nonumber\\
&=&-\lambda\io\nabla \left(\psi_{k}+\frac{1}{k}\right)^3\nabla\frac{1}{\psi_{k}+\frac{1}{k}} dx+\lambda k\io v_kdx\nonumber\\
&=&3\lambda\io|\nabla\psi_{k} |^2dx+\lambda k\io v_kdx.
\end{eqnarray}
Combining this with \eqref{js5} yields
\begin{equation}
3\lambda\io|\nabla{\psi_{k}} |^2dx+\io|\nabla v_k|^2dx=-\lambda k\io v_kdx\rightarrow\infty\ \ \mbox{as $k\rightarrow\infty$}.
\end{equation}
It does not seem to be possible to have any estimates on $\nabla v$ on the whole domain $\Omega$. Thus the sense in which the boundary condition $v=0$ on $\po$ is satisfied is an issue.

\section{Self-similar solutions}\label{sec3}

We seek a solution of the equation $\partial_t\rho+\rho^2\Delta\rho^3=0$ on $\omt$ of the form
\begin{equation}
\rho(x,t)=t^\alpha f(y), \ \ y=\frac{x}{t^\beta}.
\end{equation}
We compute
\begin{eqnarray}
\partial_t\rho&=&\alpha t^{\alpha-1}f(\frac{x}{t^\beta})-\beta t^\alpha\nabla f(y)\frac{x}{t^{\beta+1}}\nonumber\\
&=&t^{\alpha-1}\left(\alpha f(y)-\beta y\cdot\nabla f(y)\right),\\
\rho^2\Delta^2\rho^3&=&t^{5\alpha-4\beta}f^2(y)\Delta^2f^3(y).
\end{eqnarray}
Substitute these into \eqref{cr1} to arrive at
\begin{equation}
f^2(y)\Delta^2f^3(y)t^{4\alpha-4\beta+1}-\beta y\cdot\nabla f(y)+\alpha f(y)=0\ \ \mbox{on $\RD$.}
\end{equation}
Thus we need to choose $\alpha, \beta$ so that
\begin{equation}
4\alpha-4\beta+1=0.
\end{equation}
This gives \eqref{jt1}.

\begin{proof}[Proof of Theorem \ref{exis2}] As before,  we transform the fourth-order equation \eqref{jt1} into a system of two second-order equations
	\begin{eqnarray}
	\Delta f^3(y)&=&v(y)\ \ \mbox{in $\RD$},\label{jt21}\\
	\Delta v(y)&=&-\beta y\cdot\nabla\left(\frac{1}{f(y)}\right)- \frac{4\beta-1}{4}\frac{1}{f(y)}\ \ \mbox{in $\RD$}.\label{jt22}
	\end{eqnarray}
	We seek a radially symmetric solution. That is, we assume that
	\begin{equation}
	v=v(r), \ \ \ f=f(r), 
	\end{equation}
	where $r=|y|$ is the same as before.
	Then a simple calculation shows that 
	\begin{eqnarray}
	\left(f^3(r)\right)^{\prime\prime}+\frac{N-1}{r}\left(f^3(r)\right)^{\prime}&=& v(r)\ \ \mbox{in $(0,\infty)$,}\label{re1}\\
	v^{\prime\prime}(r)+\frac{N-1}{r}v^{\prime}(r)&=&-\beta r\left(\frac{1}{f(r)}\right)^\prime- \frac{4\beta-1}{4}\frac{1}{f(r)}\ \ \mbox{in $(0,\infty)$.}\label{re2}
	\end{eqnarray}
	Multiply through \eqref{re2} by $r{^{N-1}}$ to obtain
	\begin{equation}
	\left(r^{N-1}v^{\prime}(r)\right)^\prime=-\beta\left(\frac{r^N}{f(r)}\right)^\prime+\frac{(4(N-1)\beta+1)r^{N-1}}{4f(r)}.
	\end{equation}
	Integrate to yield
	\begin{equation}
	v^{\prime}(r)=-\frac{\beta r}{f(r)}+\frac{c_1}{r^{N-1}}+\frac{4(N-1)\beta+1}{4r^{N-1}}\int_{0}^{r}\frac{s^{N-1}}{f(s)}ds.
	\end{equation}
	We take the constant of integration $c_1$ to be $0$ to avoid a blow-up at $r=0$.
	Continue to integrate the preceding equation to derive
	\begin{equation}
	v(r)=c_2-\beta\int_{0}^{r}\frac{s}{f(s)}ds+\frac{4(N-1)\beta+1}{4}\int_{0}^{r}\frac{G_1(\tau,r)}{f(\tau)}d\tau,
	\end{equation}
	where 
	\begin{equation}
	G_1(\tau,r)=\int_{\tau}^{r}\frac{\tau^{N-1}}{s^{N-1}}ds=\left\{\begin{array}{ll}
	\frac{\tau(r^{N-2}-\tau^{N-2})}{(N-2)r^{N-2}}&\mbox{if $N>2$,}\\
	\tau\ln\frac{r}{\tau}&\mbox{if $N=2$.}
	\end{array}\right.
	\end{equation}
	Multiply through \eqref{re1} by $r{^{N-1}}$ and integrate the resulting equation  to deduce
	\begin{eqnarray}
	r^{N-1}\left(f^3(r)\right)^{\prime}&=&c_3+\int_{0}^{r}s^{N-1}v(s)ds\nonumber\\
	&=&c_3+c_2r^N-\beta\int_{0}^{r}s^{N-1}\int_{0}^{s}\frac{\tau}{f(\tau)}d\tau ds\nonumber\\
	&&+\frac{4(N-1)\beta+1}{4}\int_{0}^{r}\int_0^s\frac{s^{N-1}G_1(\tau, s)}{f(\tau)}d\tau ds\nonumber\\
	&=&c_3+c_2r^N-\beta\int_{0}^{r}\frac{H_1(\tau, r)}{f(\tau)}d\tau+\frac{4(N-1)\beta+1}{4}\int_{0}^{r}\frac{G_2(\tau, r)}{f(\tau)}d\tau,
	\end{eqnarray}
	where 
	\begin{eqnarray}
	H_1(\tau, r)&=&\int_{\tau}^{r}\tau s^{N-1}ds=\frac{1}{N}\tau(r^N-\tau^N),\label{h1}\\
	G_2(\tau, r)&=&\int_{\tau}^{r}s^{N-1}G_1(\tau, s)ds\nonumber\\
	&=&\left\{\begin{array}{ll}
	\frac{1}{(N-2)}\left(\frac{1}{N}\tau r^N-\frac{1}{2}r^2\tau^{N-1}+\frac{N-2}{2N}\tau^{N+1}\right)&\mbox{if $N>2$,}\\
	\frac{1}{2}\tau r^2\ln\frac{r}{\tau}-\frac{1}{4}\tau(r^2-\tau^2)&\mbox{if $N=2$}.
	\end{array}\right.\label{g2}
	\end{eqnarray}
	As before, we let $c_3=0$ to derive
	\begin{eqnarray}
	f^3(r)&=&c_4+c_2r^2-\beta\int_{0}^{r}\int_{0}^{s}\frac{H_1(\tau, s)}{s^{N-1}f(\tau)}d\tau ds\nonumber\\
	&&+\frac{4(N-1)\beta+1}{4}\int_{0}^{r}\int_{0}^{s}\frac{G_2(\tau, s)}{s^{N-1}f(\tau)}d\tau ds\nonumber\\
	&=&c_4+c_2r^2+\int_{0}^{r}\frac{G(\tau,r)}{f(\tau)}d\tau.\label{ss1}
	\end{eqnarray}
	where
	\begin{eqnarray}\label{ss2}
	G(\tau,r)&=&-\beta\int_{\tau}^{r}\frac{H_1(\tau, s)}{s^{N-1}}ds+\frac{4(N-1)\beta+1}{4}\int_{\tau}^{r}\frac{G_2(\tau, s)}{s^{N-1}} ds.
	\end{eqnarray}
	Observe that $H_1(\tau, r), G_1(\tau, r), G_2(\tau, r)$ are all non-negative for $0\leq \tau\leq r$.
	This combined with our assumption \eqref{asmp} implies
	\begin{equation}\label{gpo}
	G(\tau,r)\geq 0\ \ \mbox{for $0\leq \tau\leq r$.}
	\end{equation}
	This fact is the key to our proof.
	Set 
	\begin{equation}
	h(r)=f^3(r).
	\end{equation}
	We can write \eqref{ss1} as 
	\begin{eqnarray}
	h(r)&=&c_4+c_2r^2+\int_{0}^{r}\frac{G(\tau, r)}{h^{\frac{1}{3}}(\tau)}d\tau.
	\end{eqnarray}
	Now fix
	\begin{equation}
	c_4,\ \ \ c_2>0.
	\end{equation}Consider the function space
	\begin{equation}
	W=\{g(r)\in C[0,\infty): g(r)\geq c_4+c_2r^2 \ \ \mbox{for each $r\geq 0$}\}.
	\end{equation}
	We define an operator $T$ on $W$ as follows: For each $g\in W$ we let
	\begin{eqnarray}\label{9m2}
	T(g)&=&c_4+c_2r^2+\int_{0}^{r}\frac{G(\tau, r)}{g^{\frac{1}{3}}(\tau)}d\tau.
	\end{eqnarray}
	To see that $T$ is well-defined on $W$, we will have to separate the case where
	\begin{equation}\label{conn}
	N>2\ \ \mbox{and $N\neq 4$} 
	\end{equation}
	from the remaining case.
	Assume \eqref{conn} to be true. We calculate from \eqref{h1} and \eqref{g2} that
	\begin{eqnarray}
	G(\tau,r)&=&-\frac{\beta}{N}\int_{\tau}^{r}\frac{\tau(s^N-\tau^N)}{s^{N-1}}ds+\frac{4(N-1)\beta+1}{4(N-2)}\int_{\tau}^{r}\frac{\frac{1}{N}\tau s^N-\frac{1}{2}s^2\tau^{N-1}+\frac{N-2}{2N}\tau^{N+1}}{s^{N-1}} ds\nonumber\\
	&=&-\frac{\beta}{N}\left(\frac{1}{2}r^2\tau-\frac{N}{2(N-2)}\tau^3+\frac{1}{(N-2)}r^{-N+2}\tau^{N+1}\right)\nonumber\\
	&&+\frac{4(N-1)\beta+1}{4(N-2)}\left(\frac{1}{2N}\tau r^2+\frac{1}{2(4-N)}\tau^3-\frac{1}{2N}r^{-N+2}\tau^{N+1}\right)\nonumber\\
	&&+\frac{4(N-1)\beta+1}{8(N-2)(N-4)}r^{-N+4}\tau^{N-1}\nonumber\\
	&=&\frac{4\beta+1}{8N(N-2)}r^2\tau-\frac{12\beta+1}{8(N-2)(N-4)}\tau^3-\frac{4(N+1)\beta+1}{8N(N-2)}r^{-N+2}\tau^{N+1}\nonumber\\
	&&+\frac{4(N-1)\beta+1}{8(N-2)(N-4)}r^{-N+4}\tau^{N-1}, \label{9m1}
	\end{eqnarray}
	from whence follows
	\begin{equation}\label{9m3}
	\int_{0}^{r}\frac{G(\tau, r)}{g^{\frac{1}{3}}(\tau)}d\tau\leq c\int_{0}^{r}G(\tau, r)d\tau\leq cr^{4}\ \ \mbox{for each $g\in W$}.
	\end{equation}
	The case where $N=2$ or $4$ can be handled in an entirely similar manner. We shall omit it here.
	By virtue of \eqref{gpo}, the range of $T$ is  contained in $W$. 
	\begin{clm}\label{clm1}
		For each $R>0$ the operator $T$ has a fixed point in the space $W_R\equiv C[0,R]\cap W$. That is, there is a function $h$ in the space such that
		\begin{equation}
		h=T(h).
		\end{equation}
	\end{clm}
	\begin{proof} We wish to apply Corollary 11.2 in (\cite{GT}, p. 280). Evidently, $W_R$ is a closed convex set in $C[0,R]$ and $T$ maps $W_R$ into itself. To check that $T$ is continuous, we observe that $s^{-\frac{1}{3}}$ is uniformly Lipschitz on $[c_4, \infty)$. Let $g_1, g_2\in W_R$ be given. We estimate for $r\in [0, R]$ that
		\begin{eqnarray}
		|T(g_1)(r)-T(g_2)(r)|&\leq&\int_{0}^{r}G(\tau, r)\left|\frac{1}{g_{1}^{\frac{1}{3}}(\tau)}-\frac{1}{g_{2}^{\frac{1}{3}}(\tau)}\right|d\tau\nonumber\\
		&\leq&c\|g_1-g_{2}\|_{C[0, R]}\int_{0}^{r}G(\tau, r)d\tau\leq cR^4\|g_1-g_{2}\|_{C[0, R]}.
		\end{eqnarray}
		That is, $T$ is Lipschitz on $W_R$. To see that the range of $T$ is precompact in $C[0, R]$, for $g\in W_R$ we differentiate \eqref{9m2} to derive
		\begin{equation}\label{9t11}
		\left(T(g)\right)^\prime(r)= 2c_2r+\int_{0}^{r}\frac{\partial_rG(\tau, r)}{g^{\frac{1}{3}}(\tau)}d\tau.
		\end{equation}	
		Here we have used the fact that
		\begin{equation}
		G(r,r)=0.
		\end{equation}
		In view of \eqref{9m1} and \eqref{9m3}, we deduce
		\begin{equation}\label{9t21}
		\left|\left(T(g)\right)^\prime(r)\right|\leq 2c_2R+cR^{3}\ \ \mbox{for $r\in [0,R]$.}
		\end{equation}
		This completes the proof of the claim.
	\end{proof}
	
	Fix $R>0$ and denote by $h(r)$ the fixed point given by the above claim. We differentiate \eqref{9t11} three more times to obtain
	\begin{eqnarray}
	\left(T(g)\right)^{\prime\prime}(r)&=& 2c_2+\frac{\partial_rG(r, r)}{g^{\frac{1}{3}}(r)}+\int_{0}^{r}\frac{\partial^2_{rr}G(\tau, r)}{g^{\frac{1}{3}}(\tau)}d\tau=2c_2+\int_{0}^{r}\frac{\partial^2_{rr}G(\tau, r)}{g^{\frac{1}{3}}(\tau)}d\tau,\\
	\left(T(g)\right)^{\prime\prime\prime}(r)&=& \frac{\partial^2_{rr}G(r, r)}{g^{\frac{1}{3}}(r)}+\int_{0}^{r}\frac{\partial^3_{rrr}G(\tau, r)}{g^{\frac{1}{3}}(\tau)}d\tau=\frac{cr}{g^{\frac{1}{3}}(r)}+\int_{0}^{r}\frac{\partial^3_{rrr}G(\tau, r)}{g^{\frac{1}{3}}(\tau)}d\tau,\\
	\left(T(g)\right)^{(4)}(r)&=& \left(\frac{cr}{g^{\frac{1}{3}}(r)}\right)^\prime+\frac{\partial^3_{rrr}G(r, r)}{g^{\frac{1}{3}}(r)}+\int_{0}^{r}\frac{\partial^4_{rrrr}G(\tau, r)}{g^{\frac{1}{3}}(\tau)}d\tau.
	\end{eqnarray}
	Note from \eqref{9m1} that $\partial^3_{rrr}G(r, r)=c$ and
	\begin{equation}
	\left|\int_{0}^{r}\frac{\partial^4_{rrrr}G(\tau, r)}{g^{\frac{1}{3}}(\tau)}d\tau\right|\leq c.
	\end{equation}
	This indicates that the function $h$ lies in $C^4[0,R]\cap C^\infty(0, R]$,
	and hence $f(r)=h^{\frac{1}{3}}(r)$ is a solution to \eqref{jt1} in $B_R(0)$. 
	Observe that $\mathbf{h}\equiv(h, h^\prime,h^{\prime\prime},h^{\prime\prime\prime})$ is a bounded solution of a system of ordinary differential equations of the form $\mathbf{h}^\prime=\mathbf{F}(r,\mathbf{h})$ on $[0, R]$, where $\mathbf{F}$ is locally Lipschitz in $(0, \infty)\times(0, \infty)\times\R^3$.
	Thus we can extend $h(r)$ to $[0,\infty)$.
	The proof of Theorem \ref{exis2} is complete. \end{proof}
\begin{rem}
	It seems to be possible to find more general conditions under which $G(\tau, r)$ is non-negative. We leave this to the interested reader. The existence of a solution remains unsolved when $G(\tau, r)$ changes signs for $0\leq \tau\leq r$.
\end{rem}

\bigskip
\noindent{\bf Acknowledgment.} The author is grateful to Prof.~Jian-Guo Liu for some useful discussions during the preparation of this manuscript.
\bibliographystyle{siamplain}

\end{document}